\documentclass[titlepage,12pt]{amsart}

\usepackage[a4paper,top=2.5cm,bottom=2.60cm,left=2.8cm,right=2.8cm]{geometry} 

\usepackage{amsmath}
\usepackage{amsfonts}
\usepackage{amsbsy}
\usepackage{amssymb}
\usepackage[normalem]{ulem}
\usepackage{times}
\usepackage{mathrsfs}
\usepackage{exscale}
\usepackage{graphicx}

\allowdisplaybreaks

\usepackage[colorlinks,citecolor=blue,linkcolor=blue,
            bookmarksopen,
            bookmarksnumbered
           ]{hyperref}
           
\usepackage{tikz}
\usetikzlibrary{patterns}

\usepackage{times}

\usepackage{color}
\definecolor{gree}{rgb}   {0.,   0.66,   0.25 }
\definecolor{marin}{rgb}   {0.,   0.7,   0.7}
\definecolor{mmagenta}{rgb}{0.95,0.0,0.95}
\definecolor{orange}{rgb}   {1,   0.5,   0.}
\definecolor{brown}{rgb}   {0.4,   0.,   0.8}
\definecolor{greymg}{rgb}   {1,   0.,   0.8}
\definecolor{greygr}{rgb}   {0.5,   0.8,   0.5}

\newcommand{\Rd}{\color{red}}
\newcommand{\Bl}{\color{blue}}

\newtheorem{theorem}{Theorem}[section]
\newtheorem{lemma}[theorem]{Lemma}

\newtheorem{corollary}[theorem]{Corollary}

\theoremstyle{definition}

\theoremstyle{remark}
\newtheorem{remark}[theorem]{Remark}

\numberwithin{equation}{section}

\newcommand{\ee}{\hskip0.15ex}
\newcommand{\me}{\hskip-0.15ex}
\newcommand{\dd}[1]{_{\raise-0.6ex\hbox{$\scriptstyle #1$}}}
\newcommand{\di}{\displaystyle}
\newcommand{\on}[1]{\big|_{#1}}

\newcommand {\Norm}[2]{ \mathchoice
    {|\ee #1\ee|\dd{#2}\,}
    {| #1 |_{#2}}
    {| #1 |_{#2}}
    {| #1 |_{#2}} }
\newcommand {\DNorm}[2]{ \mathchoice
    {\|\ee #1\ee\|\dd{#2}\,}
    {\| #1 \|_{#2}}
    {\| #1 \|_{#2}}
    {\| #1 \|_{#2}} }
\newcommand {\Normc}[3]{ \mathchoice
    {|\ee #1\ee|\dd{#2}^{#3}}
    {| #1 |_{#2}^{#3}}
    {| #1 |_{#2}^{#3}}
    {| #1 |_{#2}^{#3}} }
\newcommand {\DNormc}[3]{ \mathchoice
    {\|\ee #1\ee\|\dd{#2}^{#3}}
    {\| #1 \|_{#2}^{#3}}
    {\| #1 \|_{#2}^{#3}}
    {\| #1 \|_{#2}^{#3}} }

\newcommand\bS{{\mathbb S}}
\newcommand\bP{{\mathbb P}}
\newcommand\C{{\mathbb C}}
\newcommand\R{{\mathbb R}}

\newcommand{\sA}{{\mathscr A}}
\newcommand{\sC}{{\mathscr C}}
\newcommand{\sB}{{\mathscr B}}

\newcommand{\sU}{{\mathscr U}}

\newcommand\bc{{\mathbf{c}}}

\newcommand\cM{{\mathcal{M}}}

\newcommand\cO{{\mathcal{O}}}

\newcommand\cU{{\mathcal{U}}}

\newcommand{\rP}{{\mathsf{P}}}
\newcommand{\rQ}{{\mathsf{Q}}}

\newcommand{\rd}{{\mathrm d}}

\newcommand{\sfe}{{\mathsf e}}
\newcommand{\sfE}{{\mathsf E}}
\newcommand{\sfu}{{\mathsf u}}
\newcommand{\sfU}{{\mathsf U}}

\newcommand {\gA}{\mathfrak{A}}

\newcommand {\gE}{\mathfrak{E}}
\newcommand {\gF}{\mathfrak{F}}

\newcommand {\Id}{\mathbb I}

\newcommand{\oR}{\overline{R}}

\makeatletter
\let\@setaddressesnow\@setaddresses
  \let\@setaddresses\empty
\def\@maketitle{%
  \normalfont\normalsize
  \@adminfootnotes
  \@mkboth{\@nx\shortauthors}{\@nx\shorttitle}%
  \global\topskip42\p@\relax 
  \@settitle
  \ifx\@empty\authors \else \@setauthors \fi
  \ifx\@empty\@dedicatory
  \else
    \baselineskip18\p@
    \vtop{\centering{\footnotesize\itshape\@dedicatory\@@par}%
      \global\dimen@i\prevdepth}\prevdepth\dimen@i
  \fi
  \@setabstract
  \normalsize
  \if@titlepage
  \vfill
  \@setaddressesnow
    \newpage
  \else
    \dimen@34\p@ \advance\dimen@-\baselineskip
    \vskip\dimen@\relax
  \fi
} 
\makeatother

\begin{document}

\title{Estimates for harmonic functions near pseudo-corners}

\author{Martin Costabel}
\address{Martin Costabel: \ ORCID 0000-0003-4404-9474
(Corresponding Author)}
\email{martin.costabel@univ-rennes1.fr}
\address{IRMAR UMR 6625 du CNRS, Universit\'{e} de Rennes 1 }
\address{Campus de Beaulieu,
35042 Rennes Cedex, France \bigskip}

\author{Monique Dauge}
\address{Monique Dauge: \ ORCID 0000-0003-0846-9009}
\email{monique.dauge@univ-rennes1.fr}
\address{IRMAR UMR 6625 du CNRS, Universit\'{e} de Rennes 1 }
\address{Campus de Beaulieu,
35042 Rennes Cedex, France \bigskip}





\begin{abstract}
It is well known that derivatives of solutions to elliptic boundary value problems may become unbounded near the corner of a domain with a conical singularity, even if the data are smooth. When the corner domain is approximated by more regular domains, then higher order Sobolev norms of the solutions on these domains can blow up in the limit. We study this blow-up in the simple example of the Laplace operator with Dirichlet conditions in two situations: the rounding of a corner in any dimension, and  the two-dimensional situation where a polygonal corner is replaced by two or more corners with smaller angles. We show how an inner expansion derived from a more general recent result on converging expansions into generalized power series can be employed to prove simple and explicit estimations for Sobolev norms and singularity coefficients by powers of the approximation scale.
\end{abstract}

\maketitle

{\footnotesize
\parskip 1pt
\tableofcontents
}

\section{Introduction}

We study the behavior of Sobolev norms of the solution of the Dirichlet problem for the Laplacian near a conical singularity of the boundary, when the domain $\Omega$ is approximated by a family of more regular domains $\Omega_{\varepsilon}$ that is self-similar in a neighborhood of the singular point. 
The approximating domains may be smooth, in which case Sobolev norms of any positive order can be estimated, or they may themselves have conical singularities, in which case the corresponding solutions only have finite regularity. 
In any case, it is to be expected that Sobolev norms corresponding to a regularity above the threshold defined by the first singular exponent of the conical point will not remain bounded in the limit. In Section~\ref{S:smoothpseudo} we estimate the order of this blow-up in terms of powers of the approximation scale $\varepsilon$, see Theorem~\ref{th:Wmp}.

In Section~\ref{S:polygon} we discuss in detail the case of a polygon in two dimensions with a non-convex corner, approximated by a finite number of corners with smaller, but still non-convex opening angles. Here it is interesting to study not only fractional Sobolev norms of order below the regularity threshold, but also the often practically important $H^{2}$ regularity. Namely, for $L^{2}$ right-hand sides, one can write standard decompositions of the solutions into corner singular functions and regular parts in $H^{2}$. Then one can either observe the blow-up of the coefficients of the singular functions and the $H^{2}$ norms of the regular parts or, conversely, assume that a family of right-hand sides is given such that the solutions on $\Omega_{\varepsilon}$ have $H^{2}$ regularity. With each one of these right-hand sides, the solutions on $\Omega$ will, in general, not be $H^{2}$ regular, but we show that the coefficients of the singular function of the solutions are bounded by a constant times the $L^{2}$ norm of the right-hand side, where the constant will tend to zero with a certain power of $\varepsilon$.

Our main technique is the expansion in convergent generalized power series in $\varepsilon$ proved in the recent paper \cite{CoDa2Mu:2024}. Whereas \cite{CoDa2Mu:2024} considers  convergence only in the energy norm $H^{1}$, we show here how these same expansions can be used to obtain in a simple way estimates for a range of Sobolev spaces of integer and fractional order.
In this way we can easily and explicitly analyze details of the asymptotic behavior of the solution of the Dirichlet problem on a class of singularly perturbed domains that has been studied in a more general context with techniques of asymptotic expansions. In \cite{MazNazPla:2000} it is shown how to apply those techniques to a wider class of elliptic boundary value problems, with variable coefficients that may even have singularities at the corners, whereas our use of convergent expansions confines us for the moment to the Dirichlet problem for the Laplacian, but has the advantage of simplicity and precision of the resulting estimates.

For the problem of approximation of a corner domain by smooth domains there exist studies using different techniques and different assumptions, see \cite{DanielLabrunieNistor:2024} for a recent example and a survey. Our examples are rather more general in the smoothness hypotheses for the approximating domains, but we make the restrictive assumption of a self-similar behavior. This allows us to not only precisely estimate norms that remain bounded in the limit, but also to describe the precise behavior of norms that explode in the limit.

These results provide transparent and easy to understand examples for the discussion of some aspects of potential theory on Lipschitz domains. 

In particular, concerning the question of regularity in Sobolev spaces, they illustrate how the norm of the solution operator of the Dirichlet problem can depend not only on the angles of the corners of a corner domain (i.e.\ its Lipschitz constant), but also on the distance between the corners, in contrast to the mere existence of a bounded solution operator that only depends on these angles, as long as there are only a finite number of corners.

In addition, it is seen that the choice of norm in fractional Sobolev spaces is crucial. 
In interpolation theory for Banach spaces, it has been known \cite{Ch-WHewettMoiola:2015,Ch-WHewettMoiola:2022} that for the case of Hilbert spaces the standard methods, defined for example by the $K$-functional, the $J$-functional, complex interpolation, or fractional powers of positive operators, all give the same interpolation spaces with identical norms and that for Sobolev spaces on Lipschitz domains, these norms are equivalent to the natural norms, i.e. the norms defined equivalently either by Sobolev-Slobodecki\u{\i} fractional integrals or by restriction from a larger smooth domain. However, the constants in these norm equivalences may depend on the domain in a way that is hard to estimate, especially if compounded by the notoriously difficult problem of interpolation of subspaces \cite{Asekritova:2015,Zerulla:2022}. For our family of domains $\Omega_{\varepsilon}$, we find that the interpolation norm 
may provide an operator norm that remains uniformly bounded with respect to $\varepsilon$, whereas the natural norm may lead to an operator norm that explodes as $\varepsilon\to0$. Thus we have examples where the two norms then are  equivalent, but in a non-uniform way.

\section{Preliminaries}
\subsection{Variational solutions and their regularity in smooth domains}
In this note, we address the Dirichlet problem for the Laplace operator, as a prototype of elliptic boundary value problem.  In a domain (a connected open set) $\sU$ of $\R^n$  this refers to the equations that we will simply refer to as ``Dirichlet problem''
\begin{equation}
\label{eq:1}
\begin{cases}
   \begin{array}{rl}
   \Delta u = f &\mbox{ in }\sU,\\
   u = 0 & \mbox{ on }\partial\sU.
   \end{array}
\end{cases}
\end{equation}
We consider weak solutions of problem \eqref{eq:1}, that is solutions of the variational formulation
\begin{equation}
\label{eq:2}
   u\in V,\quad \forall v\in V,\quad
   -\int_\sU \nabla u\cdot\nabla v = \int_\sU f\,v
\end{equation}
where the space $V$ incorporates for $v$ the conditions that $\nabla v\in L^2(\sU)^n$ and $v=0$ on $\partial\sU$.

If $\sU$ is {\em bounded}, one takes
\[
   V = H^1_0(\sU) \quad\mbox{closure of $\sC^\infty_0(\sU)$ in $H^1(\sU)$.}
\]
Then for any $f$ in the dual space $H^{-1}(\sU)$ of $H^1_0(\sU)$, the solution $u$ to \eqref{eq:2} exists and is unique in $H^1_0(\sU)$.

If, moreover, the domain $\sU$ has a smooth boundary, it is known since \cite{ADN:1959}, see for example \cite[Chapter 5]{Triebel78} that we have an optimal {\em regularity shift} for $u$ if $f$ is more regular in some scales of Sobolev spaces, namely 
\begin{subequations}
\begin{equation}
\label{eq:Wmp}
   \mbox{if $f\in W^{m-2,p}(\sU)$ for an integer $m\ge2$, then $u\in W^{m,p}(\sU)$,}
\end{equation}
(here we assume $m-\frac{n}{p} > 1 -\frac{n}{2}$, which ensures that $W^{m-2,p}(\sU)$ is embedded in $H^{-1}(\sU)$)
\begin{equation}
\label{eq:Hs}
   \mbox{if $f\in H^{s-2}(\sU)$ for a real $s\ge2$, then $u\in H^s(\sU)$.}
\end{equation}
\end{subequations}
Here  for $1<p<\infty$, $W^{m,p}(\sU)$ denotes the space of $v\in L^p(\sU)$ such that $\partial^\alpha u\in L^p(\sU)$ for all $|\alpha|\le m$, whereas $H^m(\sU)=W^{m,2}(\sU)$ denotes the scale of Hilbertian Sobolev spaces completed by the spaces $H^s(\sU)$ using fractional Sobolev-Slobodecki\u{\i} seminorms (see also sec.\,\ref{ss:SS}).

\subsection{Harmonic functions in regular cones}
Regular cones are models in the theory of elliptic problems in domains with conical singularities as initiated by Kondrat'ev \cite{Kond:1967}. The case of the Dirichlet Laplacian is explicitly treated by Grisvard \cite{Grisvard:1971}. For further use we revisit the bases of this study.

A {\em regular cone} $\Gamma$ is a positively homogeneous subset of $\R^n$ such that its spherical cap
\[
   \widehat\Gamma := \Gamma\cap\bS^{n-1}
\]
is open in $\bS^{n-1}$ and has a smooth, non-empty, boundary $\partial\widehat\Gamma\subset\bS^{n-1}$. We assume here that $\widehat\Gamma$ is connected.

Written in polar coordinates $(r,\vartheta)\in\R_+\times\bS^{n-1}$, the cone $\Gamma$ is isomorphic to $\R_+\times\widehat\Gamma$ and the Laplace operator associated with Dirichlet conditions on $\partial\Gamma$ becomes
\[
   \partial^2_r + (n-1)r^{-1}\partial_r + r^{-2}\Delta^{\sf dir}_{\widehat\Gamma}
\]
where $\Delta^{\sf dir}_{\widehat\Gamma}$ is the spherical Laplacian (Laplace-Beltrami operator) with Dirichlet conditions on $\partial\widehat\Gamma$.
The operator $-\Delta^{\sf dir}_{\widehat\Gamma}$ is an unbounded selfadjoint and positive operator in $L^{2}(\widehat\Gamma)$ with domain $H^2\cap H^1_0(\widehat\Gamma)$. Its eigenvalues $\mu_j$ form an unbounded sequence and the associated eigenfunctions $\psi_j$ form an orthonormal basis in $L^2(\widehat\Gamma)$. The assumptions on $\widehat\Gamma$ imply that
\begin{equation}
\label{eq:muj}
   0 < \mu_1 <  \mu_2 \le \mu_3 \le \ldots .
\end{equation}
The eigen-equation
\[
   -\Delta^{\sf dir}_{\widehat\Gamma} \psi_j = \mu_j\,\psi_j,\quad \forall j\ge1
\]
implies that any function of the form $r^\lambda\psi_j$ (with $\lambda\in\C$) satisfies
\[
   \Delta(r^\lambda\psi_j) = (\lambda^2+(n-2)\lambda - \mu_j)\psi_j\,.
\]
Introducing the roots $\lambda_j^\pm$ of the equation  $\lambda^2+(n-2)\lambda - \mu_j = 0$\,,
\begin{equation}
\label{eq:lj}
  \lambda_j^\pm = 1-\frac{n}{2} \pm \sqrt{\left(1-\frac{n}{2}\right)^2 + \mu_j}\;,
\end{equation}
we find that the functions
\begin{equation}
\label{eq:hj}
   h_j^\pm(x) := r^{\lambda_j^\pm} \psi_j(\vartheta)
\end{equation}
are harmonic in $\Gamma$, and $0$ on $\partial\Gamma$.

In dimension $n=2$, cones $\Gamma$ are plane sectors, and we can choose polar coordinates so that $\widehat\Gamma=(0,\omega)$ for some $\omega\in(0,2\pi)$. Then
$\mu_j = \left(\frac{j\pi}{\omega}\right)^2$, 
$\psi_j(\theta) = \sin\left(\frac{j\pi\theta}{\omega}\right)$, and
$\lambda^\pm_j = \pm \frac{j\pi}{\omega}$.

\subsection{Dirichlet problem in corner domains}
To simplify the exposition, we call corner domain a bounded connected open set $\Omega$ that coincides with a regular cone $\Gamma$ in a neighborhood of its vertex $0$ and is smooth everywhere else. More formally, denote by $\sB(\rho)$ the ball with center $0$ and radius $\rho$, and $\sB^\complement(\rho)$ its complement in $\R^n$.
Then $\Omega$ is such that for some $r_0>0$
\[
   \Omega\cap\sB(r_0)=\Gamma\cap\sB(r_0) \quad\mbox{and}\quad
   \partial\Omega \cap \sB^\complement(r_0/2) \;\mbox{ regular}.
\]
Let $\varphi$ be a {\em smooth cutoff} function such that
\begin{equation}
\label{eq:phi}
\varphi = 1 \quad \mbox{on }  \sB(r_0/2)\, , \qquad 
\varphi = 0 \quad \mbox{on }  \sB^\complement(r_0).
\end{equation}

The following lemma shows in which sense the functions $h^+_j$ can be considered as {\em singularities of the Dirichlet problem} and is a straightforward consequence of their degree of homogeneity.

\begin{lemma}
\label{lem:2.1}
Let $u = \varphi h^+_j$ i.e. $u(x)=\varphi(x)\, r^{\lambda^+_j}\psi_j(\vartheta)$. 
Denote by $\bP$ the space of polynomials in the variable $x$. Then
\begin{enumerate}
\item $u\in H^1_0(\Omega)$ (since $\lambda^+_j>0\ge 1-\frac{n}{2}$)
\item $\Delta u= 0$ in $\Omega\cap\sB(r_0/2)$ and $\Delta u\in \sC^\infty(\overline\Omega)$,
\item $u\in W^{m,p}(\Omega)$ \;if and only if\; 
      \{$\lambda^+_j>m-\frac{n}{p}$ \;\;or\;\; $h^+_j\in\bP$ \}.
\end{enumerate}
\end{lemma}

It is much less straightforward to prove that, when the right-hand side $f$ is regular and flat near the vertex, the functions $h^+_j$ with $\lambda^+_j<m-\frac{n}{p}$ are the only possible singular functions. We give two statements that illustrate this principle.

The first statement (which will be used in the sequel) is an expansion of harmonic functions in the vicinity of the vertex. It relies on classical expansions along the eigenbasis of the operator $\Delta^{\sf dir}_{\widehat\Gamma}$. Let us introduce the spaces
\begin{equation}
\label{eq:gF}
   \gF_\Omega := \{g \in L^2(\Omega) ,\;\; g\on{\sB(\frac{r_0}{2})}= 0\}
   \quad\mbox{and}\quad
   \gE_\Omega := \{u\in H^1_0(\Omega),\;\; \Delta u \in\gF_\Omega\}.
\end{equation} 

\begin{theorem}[\cite{CoDa2Mu:2024}]
\label{th:cj}
Let $u\in\gE_\Omega$.
Write $u(x) = \widetilde u(r,\vartheta)$ for $x\in\Omega\cap\sB(r_0)$.
Then for all $j\ge1$ and $\rho\in(0,\frac{r_0}{2}]$, the functional 
\begin{equation}
\label{eq:cjdef}
   \bc_j : u \mapsto \di\rho^{-\lambda^+_j}
   \int_{\widehat{\Gamma}} \widetilde u (\rho,\vartheta) \;\psi_j(\vartheta)\, \rd\sigma_{\vartheta}
\end{equation}
is well defined and independent of $\rho$. The sequence $\{\mathbf{c}_j(u)\}$ satisfies for all $\rho\in(0,\frac{r_0}{2}]$:
\begin{equation}
\label{eq:cj}
   \sum_{j \geq 1} \lambda^+_j \rho^{2\lambda^+_j+n-2}\,|\mathbf{c}_j(u)|^2
   = \int_{\Gamma\cap\sB(\rho)} |\nabla u|^2 \,\rd x
\end{equation}
and we have the representation formula for $u$:
\[
   u = \sum_{j \geq 1} \;\mathbf{c}_j(u) \,h^+_j \quad
   \mbox{with convergence in $H^1(\Gamma\cap{\sB(\frac{r_0}{2})})$.}
\]
\end{theorem}

The second statement is an illustration in the particular case of problem \eqref{eq:1} of very general statements that stem from the theory of elliptic boundary value problem in domains with conical singularities.
Set for $\beta\in\R$ the weighted version of the space $W^{m,p}(\Omega)$
\[
   W^{m,p}_\beta(\Omega) = \{v\in L^{p}_{\rm loc}(\Omega),\quad
   r^{|\alpha|-m+\beta} \partial^\alpha u\in L^p(\Omega),\quad |\alpha|\le m\}.
\]
Note that
\begin{itemize}
\item $\Delta$ is continuous from $W^{m,p}_\beta(\Omega)$  into $W^{m-2,p}_\beta(\Omega)$,
\item $\varphi h^+_j\in W^{m,p}_\beta(\Omega)$ \; if and only if \; $\lambda^+_j > m-\frac{n}{p}-\beta$.
\end{itemize}

\begin{theorem}[\cite{Kond:1967} for $p=2$, \cite{MazPla:1978} for general $p$]
\label{T:regsing}
Let $p>1$, $m\ge2$ and $\beta\in\R$ so that $W^{m,p}_\beta(\Omega)$ is embedded in $H^{1}(\Omega)$.
Let $u$ solution of \eqref{eq:1} with $f\in W^{m-2,p}_\beta(\Omega)$.
Assume  $\lambda^+_j \neq m-\frac{n}{p}-\beta$ for all $j\ge 1$.
Then there exist coefficients $c_j$ for $j$ such that $\lambda^+_j < m-\frac{n}{p}-\beta$ satisfying
\[
   u \; - \;\varphi\!\!\sum_{j\ge1,\; \lambda^+_j \,<\, m-\frac{n}{p}-\beta} 
   c_j\,h^+_j \in W^{m,p}_\beta(\Omega)\,.
\]
\end{theorem}

\begin{remark}
\label{R:pos}
It is not difficult to see that 
for $f\in W^{m,p}(\Omega)\cap\gF_\Omega$
the coefficients $c_{j}$ of the singular functions $h^{+}_{j}$ in the decomposition into regular and singular parts of Theorem~\ref{T:regsing} are the same as the coefficients $\mathbf{c}_j(u)$ in the eigenfunction expansion of Theorem~\ref{th:cj}. We can therefore use the representation \eqref{eq:cjdef} for the coefficients of the singular functions. As an application of this observation, we obtain, for example, that if in a neighborhood of the corner, $u$ is non-negative and not identically zero, then the coefficient $\mathbf{c}_{1}(u)$ of the first singular function is non-zero. Indeed, from the maximum principle for harmonic functions follows that for $j=1$ in the integral \eqref{eq:cjdef} both $\widetilde u $ and $\psi_1$ are positive in the interior of $\widehat\Gamma$.
\end{remark}

\section{Dirichlet problem in families of smooth pseudo-corner domains}

\subsection{Families of solutions}
The families $\big\{\Omega_\varepsilon\big\}_{\varepsilon}$ of domains under consideration are
 obtained from a corner domain $\Omega$ in which an $\varepsilon$-neighborhood of the corner $0$ is replaced by the scaled version of a pattern domain $\rP$ coinciding with the regular cone $\Gamma$ at infinity. In detail (see Figure~\ref{F:1} for an illustration), recalling that the corner domain $\Omega$ satisfies $\Omega\cap\sB(r_0) = \Gamma\cap\sB(r_0)$, we introduce the connected pattern domain $\rP\subset\R^n$ with {\em smooth boundary} such that for some $R_0>0$, we have $\rP\cap\sB^\complement(R_0) = \Gamma\cap\sB^\complement(R_0)$.
Then for any $\varepsilon\in[0,\varepsilon_0)$ where $\varepsilon_0=\frac{r_0}{R_0}$, we define $\Omega_\varepsilon$ such that
\[
   \Omega_\varepsilon\cap\sB(\varepsilon R_0) = \varepsilon\rP\cap\sB(\varepsilon R_0)
   \quad\mbox{and}\quad
   \Omega_\varepsilon\cap\sB^\complement(\varepsilon R_0) = 
   \Omega\cap\sB^\complement(\varepsilon R_0).
\]
Denoting by $\sA(\rho,r)$ the annulus of radii $\rho$ and $r$, we note that
\[
   \Omega_\varepsilon\cap\sA(\varepsilon R_0,r_0) = \Gamma\cap\sA(\varepsilon R_0,r_0).
\]
It is clear that  $\Omega_0=\Omega$ and for $\varepsilon>0$, $\Omega_\varepsilon$ has a smooth boundary.

\begin{figure}[h]
\caption{Rounded corner: \ 
(a) Domain $\Omega$, \ 
(b) Pattern $\rP$ , \
(c) $\Omega_{\varepsilon}$ (for $\varepsilon=0.2$) }
\label{F:1}
\hglue-1.5ex
\begin{minipage}{0.32\textwidth}
\centering
\begin{tikzpicture}[x=0.37\textwidth,y=0.37\textwidth]
\filldraw [fill=marin!10,draw=blue,thick](0,-1) 
  to [out=-90,in=-90] (1,-1) 
  to [out=90,in=-45] (1,1) 
  to [out=135,in=0] (-1,1) 
  to [out=180,in=180] (-1,0) 
  -- (0,0) 
  -- (0,-1) ;
\draw[draw=blue](0.0,-0.8) arc(-90:180:0.8) (-0.8,0.0);
\path (-0.8,-0.0) node[anchor=north] {\Bl{$r_{0}$}} ;
\end{tikzpicture} \\
(a)
\end{minipage} \;
\begin{minipage}{0.32\textwidth}
\centering\ \\[2.5ex]
\begin{tikzpicture}[x=0.37\textwidth,y=0.37\textwidth]
\draw[draw=red,thick] (-1.2,0) 
  -- (-0.5,0) to [out=0,in=90]  (0,-0.5) 
  -- (0,-1.2) ;
\fill[pattern=dots,pattern color=red!128] (0,-1.2) 
  -- (1.2,-1.2) 
  -- (1.2,1.2) 
  -- (-1.2,1.2) 
  -- (-1.2,0) 
  -- (-0.5,0) to [out=0,in=90]  (0,-0.5) 
  -- (0,-1.2) ;
\draw[draw=red](0.0,-0.9) arc(-90:180:0.9) (-0.9,0.0);
\path (-0.9,-0.0) node[anchor=north] {\Rd{$R_0$}} ;
\end{tikzpicture} \\[2.5ex]
(b)
\end{minipage} \;
\begin{minipage}{0.32\textwidth}
\centering
\begin{tikzpicture}[x=0.37\textwidth,y=0.37\textwidth]
\fill [fill=marin!10](0,-0.8) -- (0,-1) 
  to [out=-90,in=-90] (1,-1) 
  to [out=90,in=-45] (1,1) 
  to [out=135,in=0] (-1,1) 
  to [out=180,in=180] (-1.0,0.0) -- (-0.18,0) arc(180:-90:0.18) (0.0,-0.18) ;
\draw[draw=blue](0.0,-0.8) arc(-90:180:0.8) (-0.8,0.0);
\fill[pattern=dots,pattern color=red!128](0.0,-0.8) arc(-90:180:0.8) (-0.8,0.0)
  -- (-0.1,0) to [out=0,in=90] (0,-0.1) -- (0.0,-0.8) ;
\draw[draw=red](0.0,-0.18) arc(-90:180:0.18) (-0.18,0.0);
\draw [draw=mmagenta,thick](0,-1) 
  to [out=-90,in=-90] (1,-1) 
  to [out=90,in=-45] (1,1) 
  to [out=135,in=0] (-1,1) 
  to [out=180,in=180] (-1.0,0.0)
  -- (-0.1,0) to [out=0,in=90] (0,-0.1) 
  -- (0.0,-1.0) ;
\path (-0.8,-0.0) node[anchor=north] {\Bl{$r_0$}} ;
\path (-0.25,-0.0) node[anchor=north] {\Rd{$\varepsilon R_0$}} ;
\end{tikzpicture}\\
(c)
\end{minipage}

\end{figure}
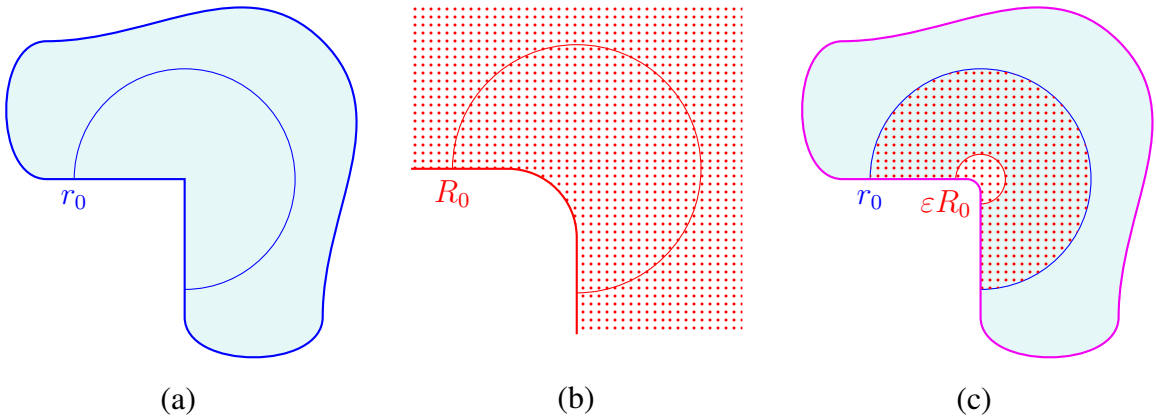

Let $\gF_\Omega$ be the space defined in \eqref{eq:gF}. For $f\in\gF_\Omega$ and $\varepsilon\in[0,\varepsilon_0)$, let $u_\varepsilon$ be the solution of the Dirichlet problem on $\Omega_\varepsilon$:
\begin{equation}
\label{eq:pbeps}
 \begin{cases}
   \begin{array}{rl}
   \Delta u_\varepsilon &= f\on{\Omega_\varepsilon} \text{ in }\Omega_\varepsilon\\
   u_\varepsilon &\in   H^1_0(\Omega_\varepsilon)\,.
  \end{array}
 \end{cases}
\end{equation}
Then $u_0$ is the solution in the limiting corner domain $\Omega$, and for any positive $\varepsilon$,
since each domain $\Omega_\varepsilon$ has a smooth boundary, the regularity shifts as stated in \eqref{eq:Wmp} and \eqref{eq:Hs} apply to $u_\varepsilon$.

\subsection{Representation of solutions as converging series}

Here we reproduce some key elements of the analysis of $u_\varepsilon$ as presented in \cite{CoDa2Mu:2024}. Representations of $u_\varepsilon$ as a converging series are obtained by starting from the  two-scale ``cross-cutoff Ansatz''
\begin{equation}
\label{eq:cross}
  u_\varepsilon(x) = 
  \Phi\Big(\frac{x}{\varepsilon}\Big) \,\sfu[\varepsilon](x)
   + \varphi(x) \,\sfU[\varepsilon]\Big(\frac{x}{\varepsilon}\Big)
\end{equation}
where the cutoff function $\varphi$ is defined in \eqref{eq:phi} and its counterpart $\Phi$ satisfies
$\Phi = 1$  in  $\sB^\complement(2R_0)$ and 
$\Phi = 0$  in $\sB(R_0)$. The unknown functions $\sfu[\varepsilon]$ and $\sfU[\varepsilon]$ are defined in $\Omega$ and $\rP$, respectively. More precisely, $\sfu[\varepsilon]$ belongs to the space $\gE_\Omega$ introduced in \eqref{eq:gF} and $\sfU[\varepsilon]$ to its counterpart $\gE_\rP$. Inserting the Ansatz \eqref{eq:cross} into the problem \eqref{eq:pbeps} leads to the equivalent block form:
\begin{equation}
\label{eq:Meps}
   \cM[\varepsilon]
   \begin{pmatrix} \sfu[\varepsilon] \\[0.3ex] \sfU[\varepsilon] \end{pmatrix}
   = \begin{pmatrix}  f \on \Omega \\ 0 \end{pmatrix}
\end{equation}
where the $2\times2$ block matrix $\cM[\varepsilon]$ has the form
\[
   \cM[\varepsilon] =
   \begin{pmatrix}
   \cM_{\Omega,\Omega} & \cM_{\Omega,\rP}[\varepsilon]\\
   \cM_{\rP,\Omega}[\varepsilon] & \cM_{\rP,\rP}
   \end{pmatrix}.
\]
Here the diagonal terms $\cM_{\Omega,\Omega}$ and $\cM_{\rP,\rP}$ are the Dirichlet Laplace operators in $\Omega$ and $\rP$, respectively, defined on the variational spaces $H^1_0(\Omega)$ and its weighted counterpart $H^1_{\beta,0}(\rP)$ on $\rP$ (with $\beta=0$). The diagonal terms are invertible. The off-diagonal terms $\cM_{\Omega,\rP}[\varepsilon]$ and $\cM_{\rP,\Omega}[\varepsilon]$ are built from commutators of the Laplace operator in slow and rapid variables with $\varphi$ and $\Phi$, respectively. 
Theorem \ref{th:cj} and its counterpart at infinity in $\rP$ allow to prove that the off-diagonal terms $\cM_{\Omega,\rP}[\varepsilon]$ and $\cM_{\rP,\Omega}[\varepsilon]$ expand into series of operators with  powers $\varepsilon^{\lambda^+_j}$ and $\varepsilon^{-\lambda^-_j}$, {\em converging} for $\varepsilon\in[0,\frac{\varepsilon_0}{4})$, see \cite[Th.\,4.11]{CoDa2Mu:2024}.

Composing equation \eqref{eq:Meps} with the inverse of the diagonal of $\cM[\varepsilon]$ yields the reduced equation
\begin{equation}
\label{e:Aeps}
   (\Id + \gA[\varepsilon]) 
   \begin{pmatrix} \sfu[\varepsilon] \\[0.3ex] \sfU[\varepsilon] \end{pmatrix} =
   \begin{pmatrix} u_0 \\[0.3ex] 0 \end{pmatrix}
\end{equation}
in which $u_0=\cM_{\Omega,\Omega}^{-1}f$. There the block matrix $\gA[\varepsilon]$ has null diagonal blocks and expands as
\begin{equation}
\label{eq:Aeps}
   \gA[\varepsilon] = \sum_{\sfe\in\sfE} \varepsilon^\sfe\,\gA_\sfe \quad\mbox{with}\quad
   \sum_{\sfe\in\sfE} \varepsilon^\sfe\, |\me|\me|\gA_\sfe |\me|\me| <\infty,\quad
   \forall\varepsilon \in[0,\frac{\varepsilon_0}{4})
\end{equation}
where the exponent set $\sfE$ is the union of the two sets $\{\lambda^+_j,\;j\ge1\}$ and $\{-\lambda^-_j,\;j\ge1\}$ 
and the norm $|\me|\me|\cdot |\me|\me|$ denotes the operator norm in $H^1_0(\Omega)\times H^1_{\beta,0}(\rP)$.
Then one can write
\begin{equation}
\label{eq:invAeps}
   \begin{pmatrix} \sfu[\varepsilon] \\[0.3ex] \sfU[\varepsilon] \end{pmatrix} =
   (\Id + \gA[\varepsilon])^{-1}
   \begin{pmatrix} u_0 \\[0.3ex] 0 \end{pmatrix}.
\end{equation}
It is proved in \cite[Sec.\,5]{CoDa2Mu:2024} that the inverse $(\Id + \gA[\varepsilon])^{-1}$ written  as a Neumann series 
\[
   (\Id + \gA[\varepsilon])^{-1} = \Id + \sum\nolimits_{k\ge1} (-1)^k(\gA[\varepsilon])^k
\]
makes sense as a formal generalized power series, and as a normally convergent series for positive $\varepsilon<\varepsilon_\star$ with small enough $\varepsilon_\star$. The exponent set attached to this inverse is the {\em monoid} $\sfE^\infty$ generated by $\sfE$: This is the set of $\{0\}$ and all finite sums of elements of $\sfE$. The outcome is that $u_\varepsilon$ is given by expression \eqref{eq:cross} in which $\sfu[\varepsilon]$ and $\sfU[\varepsilon]$ are (normally) convergent series with exponent set $\sfE^\infty$ \cite[Th.\,6.1]{CoDa2Mu:2024}. In particular
\begin{subequations}
we have the converging expansion
\begin{equation}
\label{eq:sfuepsa}
   \sfu[\varepsilon] = \sum_{\sfe\in\sfE^\infty} \varepsilon^\sfe\,\sfu_\sfe,
   \quad \sfu_\sfe\in\gE_\Omega
\end{equation}
satisfying for a constant $C$ independent of $\varepsilon$ and $f$
\begin{equation}
\label{eq:sfuepsb}
   \sum_{\sfe\in\sfE^\infty} \varepsilon^\sfe\,\DNorm{\sfu_\sfe}{H^1(\Omega)} 
   \le C \DNorm{f}{L^2(\Omega)}\quad
   \forall\varepsilon\in[0,\varepsilon_\star).
\end{equation}
\end{subequations}

Rather than directly using the representation \eqref{eq:cross} in which the coefficients $\sfu_\sfe$ depend on the cutoff functions, we will base our further analysis on the related intrinsic ``inner expansion'' given in \cite[Th.\,6.10]{CoDa2Mu:2024}. This statement requires the introduction of the canonical functions in $\rP$, cf \cite[Lemma\,6.3]{CoDa2Mu:2024}, which in spite of their appearance do not depend on the cutoff function $\Phi$ :
\begin{equation}
\label{eq:Kj}
   K^+_j = \Phi h^+_j - \cM_{\rP,\rP}^{-1}\big(\Delta(\Phi h^+_j)\big),\quad j\ge1.
\end{equation}
By construction, the $K^+_j$ are non-zero harmonic functions in $\rP$ satisfying homogeneous Dirichlet conditions on $\partial\rP$. We reproduce \cite[Lemma 6.7]{CoDa2Mu:2024} which we will use in the sequel:

\begin{lemma}
Let $R\ge 2R_0$ and $j\ge1$.
Then $K^+_j$ satisfies the following bounds in $\rP\cap\sB(R)$, with a constant $C=C_{\eqref{eq:H1Kj}}$ independent of $j$ and $R$:
\begin{equation}
\label{eq:H1Kj}
   R \DNorm{\nabla K^+_j}{L^2(\rP\cap\sB(R))} + \DNorm{K^+_j}{L^2(\rP\cap\sB(R))} \le C
   (\lambda^+_j+1)^{1/2}\, R^{\lambda^+_j+\frac{n}{2}} \,.
\end{equation}
\end{lemma}

The inner expansion can be stated as
\begin{theorem}\label{th:uepsinner} 
Let $f\in\gF_\Omega$. Then the solution $u_\varepsilon$ of problem \eqref{eq:pbeps} is given in $\Omega_\varepsilon\cap\sB(r_0/2)$ by the following expansion converging in the $H^1$ norm for all $0<\varepsilon<\varepsilon_\star$
\begin{equation}\label{uepsinner.eq1}
   u_\varepsilon(x) = 
   \sum_{j\ge 1} \sum_{\sfe\in\sfE^\infty} \varepsilon^{\sfe+\lambda^+_j}
   \mathbf{c}_j(\sfu_\sfe)\, K^+_j\!\left(\frac{x}{\varepsilon}\right)\,,
   \quad  x\in\Omega_\varepsilon\cap\sB(r_0/2).
\end{equation}
Here, $\sfu_\sfe$ comes from \eqref{eq:sfuepsa}, $\mathbf{c}_j$ from Theorem \ref{th:cj}, and $K^+_j$ from \eqref{eq:Kj}.
\end{theorem}

Note that $u_{0}=\sfu_{0}$, that is, $u_{\varepsilon}$ for $\varepsilon=0$ coincides with $\sfu_{\sfe}$ for $\sfe=0$, because both satisfy the same Dirichlet problem \eqref{eq:pbeps} on $\Omega_{0}=\Omega$.

\section{Blow-up of Sobolev norms in smooth pseudo-corner domains}
\label{S:smoothpseudo}

Let us recall that the spherical cap $\widehat\Gamma$ of the cone $\Gamma$ is supposed to be connected, and has a smooth boundary in $\bS^{n-1}$. The pattern domain $\rP$ is assumed to be connected and to have a smooth boundary. As we include the case where $\Gamma$ is a half-space (hence is smooth!) we complete these assumptions by the requirement that $\rP$ does not coincide with $\Gamma$. 

Choose $R'_0:=\frac{3}{2}R_0$ and set
\begin{equation}
\label{eq:Q}
   \rQ := \rP\cap\sB(R'_0)
\end{equation}
so that $\rQ$ is certainly non-empty and is sitting at the location of the perturbation.

\subsection{Integer order Sobolev spaces}
As usual the seminorm and norm in the space $W^{m,p}(\cU)$ are defined according to
\[
   \Normc{v}{W^{m,p}(\sU)}{p} = \sum_{|\alpha|=m} \DNormc{\partial^\alpha v}{L^p(\sU)}{p}
   \quad\mbox{and}\quad
   \DNormc{v}{W^{m,p}(\sU)}{p} = \sum_{|\alpha|\le m} \DNormc{\partial^\alpha v}{L^p(\sU)}{p}\,.
\]

\begin{theorem}
\label{th:Wmp}
Let $p>1$ and $m\ge2$ with $m-\frac{n}{p}>1-\frac{n}{2}$ so that $W^{m,p}(\Omega_\varepsilon)\subset H^1(\Omega_\varepsilon)$. Let $f\in\gF_\Omega\cap W^{m-2,p}(\Omega)$. Then for any $\varepsilon\in(0,\varepsilon_0)$, the solution $u_\varepsilon$ of problem \eqref{eq:pbeps} belongs to $W^{m,p}(\Omega_\varepsilon)$. Moreover, for any $\varepsilon\in(0,\varepsilon_\star]$ it 
satisfies the inequality
\begin{equation}
\label{eq:ueWmpQ}
   \Norm{u_\varepsilon}{W^{m,p}(\varepsilon\rQ)} \ge 
   \varepsilon^{\lambda^+_1+\frac{n}{p}-m} \left( 
   |\mathbf{c}_1(u_0)|\, \Norm{K^+_1}{W^{m,p}(\rQ)}
   - \gamma_0\,\varepsilon^{\lambda'} \DNorm{f}{L^2(\Omega)}
   \right)
\end{equation}
where
\begin{itemize}
\item[{\em(i)}] $\lambda'=\min\{\lambda^+_1,\lambda^+_2-\lambda^+_1\}>0$,
\item[{\em(ii)}] $\Norm{K^+_1}{W^{m,p}(\rQ)}>0$,
\item[{\em(iii)}] $\gamma_0$ is independent of $f$ and $\varepsilon$.
\end{itemize} 
The lower bound \eqref{eq:ueWmpQ} implies a blow-up of the seminorm of $u_\varepsilon$ in $W^{m,p}(\Omega_\varepsilon)$ if the two following conditions \eqref{eq:c1} and \eqref{eq:lam1} are satisfied:
\begin{equation}
\label{eq:c1}
   \mathbf{c}_1(u_0) \neq0
\end{equation} 
and
\begin{equation}
\label{eq:lam1}
   \lambda^+_1<m-\frac{n}{p}\,.
\end{equation}
\end{theorem}

\begin{remark}
Condition \eqref{eq:lam1} implies that if $\mathbf{c}_1(u_0) \neq0$, the solution $u_0$ of the limit problem does not belong to $W^{m,p}(\Omega)$.
\end{remark}

\begin{remark}
The region $\varepsilon\rQ$ is a boundary layer (in the form of a corner layer) that is compactly embedded in the region where $f$ is assumed to be $0$. Hence the regularity of $u_\varepsilon$ in $\varepsilon\rQ$ is independent of the regularity of $f$ elsewhere.
Indeed, following some of the lines of the proof of the lower bound \eqref{eq:ueWmpQ}, one can also find an upper bound
\begin{equation}
\label{eq:ueWmpQu}
   \Norm{u_\varepsilon}{W^{m,p}(\varepsilon\rQ)} \le 
   \varepsilon^{\lambda^+_1+\frac{n}{p}-m} \left( 
   |\mathbf{c}_1(u_0)|\, \Norm{K^+_1}{W^{m,p}(\rQ)}
   + \gamma_0\,\varepsilon^{\lambda'} \DNorm{f}{L^2(\Omega)}
   \right).
\end{equation}
This is true even if $f$ has only $L^{2}$ regularity.

On the other hand, the condition \eqref{eq:c1} leading to blow-up  can occur for $f\in\gF_\Omega\cap\sC^\infty(\overline\Omega)$. An example is $f=\Delta(\varphi h^+_1)$, cf Lemma \ref{lem:2.1}.
\end{remark}

\begin{remark}
The blow-up estimate \eqref{eq:ueWmpQ} is valid including in the case when $\Gamma$ is a half-space. In this case, $\lambda^+_1=1$ and $\lambda^+_2=2$, and the inequality \eqref{eq:ueWmpQ} becomes
\[
   \Norm{u_\varepsilon}{W^{m,p}(\varepsilon\rQ)} \ge 
   \varepsilon^{1+\frac{n}{p}-m} \left( 
   |\mathbf{c}_1(u_0)|\, \Norm{K^+_1}{W^{m,p}(\rQ)}
   - \gamma_0\,\varepsilon \DNorm{f}{L^2(\Omega)}
   \right)\,.
\]
Let $\Gamma$ be the half-space $\{x_{n}>0\}$. Then the ``singular functions'' $h_{j}^{+}$ are harmonic polynomials in $(x_{1},\dots,x_{n})$ vanishing at $x_{n}=0$ with $h_{1}^{+}(x)=x_{n}$, and the expansion of Theorem~\ref{th:cj} is simply the Taylor expansion of a harmonic function vanishing at $x_{n}=0$.
Therefore condition \eqref{eq:c1} is the local condition $\partial_nu_0(O)\neq0$.
We can summarize the conclusion of Theorem \ref{th:Wmp} for this case: A self-similar perturbation around a point $O$ in a flat part of the boundary gives rise to a blow-up of $W^{m,p}$ norms when $p>n/(m-1)$
as soon as the normal derivative of $u_{0}$ at $O$ does not vanish.
\end{remark}

\begin{proof}[Proof of Theorem \ref{th:Wmp}.]\mbox{ }\\
{\em Step 1.} Use Theorem \ref{th:uepsinner} to write
\[
   \Norm{u_\varepsilon}{W^{m,p}(\varepsilon\rQ)} \ge 
   \varepsilon^{\lambda^+_1}
   |\mathbf{c}_1(\sfu_0)|\, 
   \Norm{K^+_1\big(\tfrac{\cdot}{\varepsilon}\big)}{W^{m,p}(\varepsilon\rQ)} -
   \underbrace{\sum_{j\ge 1} \sum_{\sfe\in\sfE^\infty}}_{(j,\sfe)\neq(1,0)} 
   \varepsilon^{\sfe+\lambda^+_j} 
   |\mathbf{c}_j(\sfu_\sfe)|\, 
   \Norm{K^+_j\big(\tfrac{\cdot}{\varepsilon}\big)}{W^{m,p}(\varepsilon\rQ)}.
\]
Note that we have for all $j\ge1$ by change of variables
\begin{equation}
\label{eq:varch}
   \Norm{K^+_j\big(\tfrac{\cdot}{\varepsilon}\big)}{W^{m,p}(\varepsilon\rQ)} =
   \varepsilon^{-m+\frac{n}{p}}
   \Norm{K^+_j}{W^{m,p}(\rQ)}
\end{equation}
hence the lower bound of $\Norm{u_\varepsilon}{W^{m,p}(\varepsilon\rQ)}$ by
\begin{equation}
\label{eq:lb}
   \varepsilon^{\lambda^+_1-m+\frac{n}{p}}
   |\mathbf{c}_1(\sfu_0)|\, 
   \Norm{K^+_1}{W^{m,p}(\rQ)}  -
   \underbrace{\sum_{j\ge 1} \sum_{\sfe\in\sfE^\infty}}_{(j,\sfe)\neq(1,0)} 
   \varepsilon^{\sfe+\lambda^+_j-m+\frac{n}{p}} 
   |\mathbf{c}_j(\sfu_\sfe)|\, 
   \Norm{K^+_j}{W^{m,p}(\rQ)}.
\end{equation}

{\em Step 2.} Let us prove that the coefficient $\Norm{K^+_1}{W^{m,p}(\rQ)}$ occurring in \eqref{eq:lb} is positive. 
Would this seminorm be zero, the harmonic function $K^+_1$ would be an homogeneous polynomial in $\rQ$, hence in $\rP$ by connectedness. Hence its nodal set would be homogeneous, thus form a cone, which contradicts the assumption that $\rP\neq\Gamma$.

{\em Step 3.} Let us prove a uniform upper bound for $\Norm{K^+_j}{W^{m,p}(\rQ)}$.
For this we rely on classical elliptic a priori estimates in nested subdomains of a domain with smooth boundary: Let
\[
   0<\oR<\oR'.
\]
Then there exists a constant $C=C[\rP,\oR,\oR',m,p]$ such that for any $u\in H^1(\rP\cap\sB(\oR'))$ satisfying Dirichlet boundary condition $u=0$ on $\partial\rP\cap\sB(\oR')$
\begin{equation}
\label{estBRR'}
   \Norm{u}{W^{m,p}(\rP\cap\sB(\oR))} \le 
   C\big( \DNorm{\Delta u}{W^{m-2,p}(\rP\cap\sB(\oR'))} 
   + \DNorm{u}{H^{1}(\rP\cap\sB(\oR'))} \big)\,.
\end{equation}
Choose $\oR=R'_0=\frac{3}{2}R_0$ and $\oR'=2R_0$.
We note that $\Delta K^+_j=0$ and $K^+_j=0$ on $\partial\rP$ by construction. 
Hence \eqref{estBRR'} implies the estimates
\begin{equation}
\label{eq:estKj}
   \Norm{K^+_j}{W^{m,p}(\rQ)} \le 
   C \DNorm{K^+_j}{H^{1}(\rP\cap\sB(2R_0))} 
\end{equation}
where the constant $C=C_{\eqref{eq:estKj}}$ {\em does not depend on $j$}. Using the control of the $H^1$ norms given by \eqref{eq:H1Kj} for $R=2R_0$ we deduce
\begin{equation}
\label{eq:WmpKj}
   \Norm{K^+_j}{W^{m,p}(\rQ)} \le 
   C (\lambda^+_j+1)^{1/2}\, (2R_0)^{\lambda^+_j+\frac{n}{2}} \,
\end{equation}
where $C=C_{\eqref{eq:WmpKj}}$ can be defined as $\max\{1,\frac{1}{2R_0}\}C_{\eqref{eq:H1Kj}}C_{\eqref{eq:estKj}}$, so does not depend on $j$.

{\em Step 4.} Upper bound on the coefficients $|\mathbf{c}_j(\sfu_\sfe)|$. Since each $\sfu_\sfe$ belongs to the space $\gE_\Omega$ by construction, by \eqref{eq:cj} we find immediately
\[
   \sqrt{\lambda^+_j} \,\left(\frac{r_0}{2}\right)^{\lambda^+_j+\frac{n}{2}-1}
   \,|\mathbf{c}_j(\sfu_\sfe)| \le 
   \DNorm{\sfu_\sfe}{H^1(\Omega)},
   \quad\forall j\ge1,\;\forall\sfe\in\sfE^\infty\,.
\]

{\em Step 5.} Upper bound for the terms on the right part of \eqref{eq:lb}. Combining \eqref{eq:WmpKj} with the previous inequality we obtain
\begin{equation}
\label{eq:ub}
   \varepsilon^{\sfe+\lambda^+_j-m+\frac{n}{p}} 
   |\mathbf{c}_j(\sfu_\sfe)|\, 
   \Norm{K^+_j}{W^{m,p}(\rQ)} \le 
   C\,\varepsilon^{\lambda^+_j+\frac{n}{p} - m} 
   \Big(\frac{4R_0}{r_0} \Big)^{\lambda^+_j} 
   \varepsilon^{\sfe} \DNorm{\sfu_\sfe}{H^1(\Omega)} 
\end{equation}
where the constant $C=C_{\eqref{eq:ub}}$ can be defined as (recall that $\lambda^+_1>0$)
\[
   C_{\eqref{eq:ub}} = C_{\eqref{eq:WmpKj}} 
   \frac{r_0}{2}\Big(\frac{4R_0}{r_0} \Big)^{\frac{n}{2}} 
   \max_{j\ge1} \Big(\frac{\lambda^+_j+1}{\lambda^+_j}\Big)^{1/2} =
   C_{\eqref{eq:WmpKj}} 
   \frac{r_0}{2}\Big(\frac{4R_0}{r_0} \Big)^{\frac{n}{2}}
   \Big(\frac{\lambda^+_1+1}{\lambda^+_1}\Big)^{1/2} \,.
\]
{\em Step 6.} Upper bound for the right part of \eqref{eq:lb}. We write this double sum in two packets
\[
   \underbrace{\sum_{j\ge 1} \sum_{\sfe\in\sfE^\infty}}_{(j,\sfe)\neq(1,0)} 
   \varepsilon^{\sfe+\lambda^+_j-m+\frac{n}{p}} 
   |\mathbf{c}_j(\sfu_\sfe)|\, 
   \Norm{K^+_j}{W^{m,p}(\rQ)}
  = \Sigma_1 + \Sigma_2   
\]
with
\[
   \Sigma_1 =   \sum_{\sfe\in\sfE^\infty\!,\;\sfe\neq0} \varepsilon^{\sfe+\lambda^+_1-m+\frac{n}{p}} 
   |\mathbf{c}_1(\sfu_\sfe)|\, 
   \Norm{K^+_1}{W^{m,p}(\rQ)}
\]
and
\[
   \Sigma_2 =    \sum_{j\ge2} \sum_{\sfe\in\sfE^\infty} \varepsilon^{\sfe+\lambda^+_j-m+\frac{n}{p}} 
   |\mathbf{c}_j(\sfu_\sfe)|\, 
   \Norm{K^+_j}{W^{m,p}(\rQ)}\,.
\]
With $C=C_{\eqref{eq:S1}} := C_{\eqref{eq:ub}} \big(\frac{4R_0}{r_0} \big)^{\lambda^+_1}$, we have
\begin{equation}
\label{eq:S1}
   \Sigma_1 \le C \, \varepsilon^{\lambda^+_1+\frac{n}{p} - m} 
   \sum_{\sfe\in\sfE^\infty\!,\;\sfe\neq0} \varepsilon^{\sfe} \DNorm{\sfu_\sfe}{H^1(\Omega)}
\end{equation}
while with $C=C_{\eqref{eq:S2}} := C_{\eqref{eq:ub}}$
\begin{multline}
\label{eq:S2}
   \Sigma_2 \le C \, \sum_{j\ge2} \sum_{\sfe\in\sfE^\infty} 
   \varepsilon^{\lambda^+_j+\frac{n}{p} - m} 
   \Big(\frac{4R_0}{r_0} \Big)^{\lambda^+_j} 
   \varepsilon^{\sfe} \DNorm{\sfu_\sfe}{H^1(\Omega)} = \\
   C \varepsilon^{\frac{n}{p} - m}
   \left( \sum_{j\ge2}    \varepsilon^{\lambda^+_j} 
   \Big(\frac{4R_0}{r_0} \Big)^{\lambda^+_j} \right)
   \left( \sum_{\sfe\in\sfE^\infty} \varepsilon^{\sfe} \DNorm{\sfu_\sfe}{H^1(\Omega)}\right)\,.
\end{multline}

\begin{lemma}
\label{lem:S1}
Let $\xi\in(0,\max\{1,\varepsilon_\star\})$. There exists $\gamma_1>0$ such that
\[
   \sum_{\sfe\in\sfE^\infty\!,\;\sfe\neq0} \varepsilon^{\sfe} \DNorm{\sfu_\sfe}{H^1(\Omega)}
   \le \gamma_1\, \varepsilon^{\lambda^+_1}
   {\DNorm{f}{L^2(\Omega)}},\quad\forall\varepsilon\in[0,\xi].
\]
\end{lemma}

\begin{proof}
Recall that $\sfE^\infty\setminus\{0\}$ is the set of all finite sums of the $\pm\lambda^\pm_j$ defined in \eqref{eq:lj}. The smallest element of $\{\pm\lambda^\pm_j,\;j\ge1\}$ is $\lambda^+_1$. Hence it is also the smallest element of $\sfE^\infty\setminus\{0\}$. Hence we can write
\begin{equation}
\label{eq:eneq0}
   \sum_{\sfe\in\sfE^\infty\!,\;\sfe\neq0} \varepsilon^{\sfe} \DNorm{\sfu_\sfe}{H^1(\Omega)} =
   \varepsilon^{\lambda^+_1} 
   \sum_{\sfe\in\sfE^\infty\setminus\{0\}} \varepsilon^{\sfe-\lambda^+_1} 
   \DNorm{\sfu_\sfe}{H^1(\Omega)}.
\end{equation}
Let $\delta>0$ be such that $\xi^{1-\delta}<\varepsilon_\star$. 
\begin{itemize}
\item There is a finite number of elements $\sfe\in\sfE^\infty\setminus\{0\}$ such that $\sfe<\delta^{-1}\lambda^+_1$.

\item For the other exponents $\sfe$, we have $\delta\sfe\ge \lambda^+_1$, hence if $\varepsilon\le1$
\[
   \varepsilon^{\sfe-\lambda^+_1} = 
   \varepsilon^{\sfe(1-\lambda^+_1/\sfe)} \le 
   \varepsilon^{\sfe(1-\delta)}=
   (\varepsilon^{1-\delta})^\sfe.
\]
\end{itemize}
Hence for $\varepsilon\le1$
\[
\begin{aligned}
   \sum_{\sfe\in\sfE^\infty\setminus\{0\}} \varepsilon^{\sfe-\lambda^+_1} 
   \DNorm{\sfu_\sfe}{H^1(\Omega)} &=
   \sum_{\lambda^+_1\le\sfe<\delta^{-1}\lambda^+_1}
   \varepsilon^{\sfe-\lambda^+_1} 
   \DNorm{\sfu_\sfe}{H^1(\Omega)} +
   \sum_{\sfe\ge\delta^{-1}\lambda^+_1}
   \varepsilon^{\sfe-\lambda^+_1} 
   \DNorm{\sfu_\sfe}{H^1(\Omega)} \\
&\le
   \sum_{\lambda^+_1\le\sfe<\delta^{-1}\lambda^+_1}
   \varepsilon^{\sfe-\lambda^+_1} 
   \DNorm{\sfu_\sfe}{H^1(\Omega)} +
   \sum_{\sfe\ge\delta^{-1}\lambda^+_1}
   (\varepsilon^{1-\delta})^\sfe 
   \DNorm{\sfu_\sfe}{H^1(\Omega)} .
\end{aligned}
\]
Therefore, if $\varepsilon\le\xi$, thanks to \eqref{eq:sfuepsb} the right-hand side defines a quantity bounded independently of $\varepsilon$, which, owing to \eqref{eq:eneq0}, proves the lemma.
\end{proof}

\begin{lemma}
\label{lem:S2}
Let $\xi<\frac{r_0}{4R_0}$. There exists $\gamma_2>0$ such that 
\[
   \sum_{j\ge2}    \varepsilon^{\lambda^+_j} 
   \Big(\frac{4R_0}{r_0} \Big)^{\lambda^+_j} \le 
   \gamma_2\, \varepsilon^{\lambda^+_2},\quad \forall\varepsilon\in[0,\xi].
\]
\end{lemma}

\begin{proof}
We write
\[
   \sum_{j\ge2}    \varepsilon^{\lambda^+_j} 
   \Big(\frac{4R_0}{r_0} \Big)^{\lambda^+_j} =
   \varepsilon^{\lambda^+_2} \Big(\frac{4R_0}{r_0} \Big)^{\lambda^+_2}
   \sum_{j\ge2}  \varepsilon^{\lambda^+_j-\lambda^+_2} 
   \Big(\frac{4R_0}{r_0} \Big)^{\lambda^+_j-\lambda^+_2}   
\]
and use \cite[Lemma\,4.14]{CoDa2Mu:2024} (which relies on Weyl's law for the distribution of the eigenvalues $\mu_j$ of the operator $-\Delta^{\sf dir}_{\widehat\Gamma}$).
\end{proof}

Putting together \eqref{eq:S1} and \eqref{eq:S2}, and using Lemmas \ref{lem:S1}, \ref{lem:S2}, and the convergence estimate \eqref{eq:sfuepsb} we find
\begin{equation}
\label{eq:S1S2}
    \Sigma_1+\Sigma_2 \le \varepsilon^{\lambda^+_1+\frac{n}{p}-m}
   \big(\gamma'_1\,\varepsilon^{\lambda^+_1} + 
   \gamma'_2\, \varepsilon^{\lambda^+_2-\lambda^+_1}\big){\DNorm{f}{L^2(\Omega)}},
   \quad \forall\varepsilon\in[0,\xi].
\end{equation}
for some constants $\gamma'_1$ and $\gamma'_2$ independent of $\varepsilon$.

{\em Step 7.} Conclusion of the proof of Theorem \ref{th:Wmp}. Formulas \eqref{eq:lb} and \eqref{eq:S1S2} yield
\begin{equation*}
   \Norm{u_\varepsilon}{W^{m,p}(\varepsilon\rQ)} \ge 
   \varepsilon^{\lambda^+_1+\frac{n}{p}-m} \left( 
   |\mathbf{c}_1(\sfu_0)|\, \Norm{K^+_1}{W^{m,p}(\rQ)}
   - \big(\gamma'_1\,\varepsilon^{\lambda^+_1} + 
   \gamma'_2\, \varepsilon^{\lambda^+_2-\lambda^+_1}\big){\DNorm{f}{L^2(\Omega)}}
   \right).
\end{equation*}
Since the connectedness of $\widehat\Gamma$ implies that $\lambda^+_2>\lambda^+_1$, the Theorem is proved.
\end{proof}

\subsection{Extension to fractional order Sobolev spaces}
\label{ss:SS}
The result of Theorem \ref{th:Wmp} extends to other function spaces, in particular Sobolev spaces defined by H\"older type seminorms, called Sobolev-Slobodecki\u{\i} seminorms: For $\theta\in(0,1)$ and $p\ge1$, the $W^{\theta,p}(\sU)$ seminorm is defined by
\[
   \Normc{v}{W^{\theta,p}(\sU)}{p} = 
   \int_\sU\int_\sU \frac{|v(x)-v(y)|^p}{|x-y|^{\theta p+n}}\,\rd x\rd y
\]
and for non-integer $s>0$, the norm of $W^{s,p}(\sU)$ is 
\[
   \DNormc{v}{W^{s,p}(\sU)}{p} = 
   \DNormc{v}{W^{\lfloor s \rfloor,p}(\sU)}{p} +
   \sum_{|\alpha|= \lfloor s \rfloor} \Normc{\partial^\alpha v}{W^{s-\lfloor s \rfloor,p}(\sU)}{p}\,.
\]
It is easy to see that the identity \eqref{eq:varch} concerning the change of variables $x\mapsto\frac{x}{\varepsilon}$ extends to fractional seminorms: For any $K\in W^{s,p}(\sU)$ there holds
\[
   \Norm{K\big(\tfrac{\cdot}{\varepsilon}\big)}{W^{s,p}(\varepsilon\sU)} =
   \varepsilon^{-s+\frac{n}{p}}
   \Norm{K}{W^{s,p}(\sU)}\,.
\]
This shows that the blow-up estimates \eqref{eq:ueWmpQ} remain valid for fractional order Sobolev spaces $W^{s,p}(\Omega)$ if one replaces $m-\frac{n}{p}$ by $s-\frac{n}{p}$ everywhere.

\section{Dirichlet problem in polygons with approaching corners}
\label{S:polygon}
{From now on we are in dimension $n=2$ and all considered domains $\sU$ are polygons, which means that their boundaries are a finite union of segments. For simplicity we assume that these domains are Lipschitz, which means that their openings $\omega_j$ at each of their corners belongs to $(0,\pi)\cup(\pi,2\pi)$.  We will evaluate the regularity of solutions in the scale of Hilbertian Sobolev spaces of fractional order $H^s(\sU)=W^{s,2}(\sU)$.

We exhibit families $\Omega_\varepsilon$ of polygons in which several corners are approaching each other, giving rise to blow-up of Sobolev norms and several distinct realizations of solution operators.

\subsection{Blow up of Sobolev norms in polygons with approaching corners}
\label{ss:AppCor}
To fix ideas, we assume that the limiting domain is a polygon with only one non-convex corner placed at the origin $O$ and with opening $\omega\in(\pi,2\pi)$. Hence the cone $\Gamma$ is a sector of opening $\omega$. We assume that the pattern domain $\rP$ is polygonal too, which means that it has a piecewise affine boundary with a finite number $L\ge2$ of corners $O_\ell\in\partial\rP$, $\ell=1,\ldots,L$, with openings $\varpi_\ell$. We assume for simplicity that all $\varpi_\ell$ are non-convex. 
Then the $\Omega_\varepsilon$ are polygons with non-convex corners at points ${\textsc x}_\ell=\varepsilon O_\ell$. 

Considering the adjacent triangles $O,\,O_\ell,\,O_{\ell+1}$ for $\ell=1,\ldots,L-1$, one finds
\begin{equation}
\label{eq:varpi}
   \sum_{\ell=1}^L  (\varpi_\ell - \pi) = \omega -\pi,
\end{equation}
which implies that $\varpi_\ell<\omega$ for all $\ell=1,\ldots,L$. Setting
\begin{equation}
\label{eq:varpimax}
   \varpi := \max_{\ell=1}^L\{\varpi_\ell\}
\end{equation}
we see that
\begin{equation}
\label{eq:varpipi}
   \varpi < \omega.
\end{equation}

We illustrate this in Figure~\ref{F:2} with an example for which 
 $\omega=\frac{3\pi}{2}$ and $\varpi_\ell=\frac{5\pi}{4}$ for $\ell=1,2$.

\begin{figure}[h]
\caption{Cut corner: \ 
(a) Domain $\Omega$, \ 
(b) Pattern $\rP$ , \
(c) $\Omega_{\varepsilon}$ (for $\varepsilon=0.4$) }
\label{F:2}
\hglue-1.5ex
\begin{minipage}{0.32\textwidth}
\centering\ \\[2ex]
\begin{tikzpicture}[x=0.37\textwidth,y=0.37\textwidth]
\filldraw [fill=marin!10,draw=blue,thick](0,-1.1) 
  -- (1.1,-1.1) 
  -- (1.1,1.1) 
  -- (-1.1,1.1) 
  -- (-1.1,0) 
  -- (0,0) 
  -- (0,-1.1) ;
\draw[draw=blue](0.0,-0.8) arc(-90:180:0.8) (-0.8,0.0);
\path (-0.8,-0.0) node[anchor=north] {\Bl{$r_{0}$}} ;
\end{tikzpicture} \\[1ex]
(a)
\end{minipage} \;
\begin{minipage}{0.32\textwidth}
\centering\ \\[2.5ex]
\begin{tikzpicture}[x=0.37\textwidth,y=0.37\textwidth]
\draw[draw=red,thick] (-1.2,0) -- (-0.5,0) --  (0,-0.5) -- (0,-1.2) ;
\fill[pattern=dots,pattern color=red!128] (0,-1.2) -- (1.2,-1.2) -- (1.2,1.2) -- (-1.2,1.2) -- (-1.2,0) -- (-0.5,0) -- (0,-0.5) --(0,-1.2) ;
\draw[draw=red](0.0,-0.9) arc(-90:180:0.9) (-0.9,0.0);
\path (-0.9,-0.0) node[anchor=north] {\Rd{$R_0$}} ;
\end{tikzpicture} \\
(b)
\end{minipage} \;
\begin{minipage}{0.32\textwidth}
\centering\ \\[2ex]
\begin{tikzpicture}[x=0.37\textwidth,y=0.37\textwidth]
\fill [fill=marin!10](0,-0.8) -- (0,-1.1) 
  -- (1.1,-1.1) 
  -- (1.1,1.1) 
  -- (-1.1,1.1) 
  -- (-1.1,0.0)
  -- (-0.3,0) arc(180:-90:0.3) (0.0,-0.3) ;
\draw[draw=blue](0.0,-0.8) arc(-90:180:0.8) (-0.8,0.0);
\fill[pattern=dots,pattern color=red!128](0.0,-0.8) arc(-90:180:0.8) (-0.8,0.0)
  -- (-0.2,0) -- (0,-0.2) -- (0.0,-0.8) ;
\draw[draw=red](0.0,-0.3) arc(-90:180:0.3) (-0.3,0.0);
\draw [draw=mmagenta,thick](0,-1.1) 
  -- (1.1,-1.1) 
  -- (1.1,1.1) 
  -- (-1.1,1.1) 
  -- (-1.1,0.0)
  -- (-0.2,0) -- (0,-0.2) 
  -- (0.0,-1.1) ;
\path (-0.8,-0.0) node[anchor=north] {\Bl{$r_0$}} ;
\path (-0.3,-0.0) node[anchor=north] {\Rd{$\varepsilon R_0$}} ;
\end{tikzpicture}\\[2ex]
(c)
\end{minipage}

\end{figure}
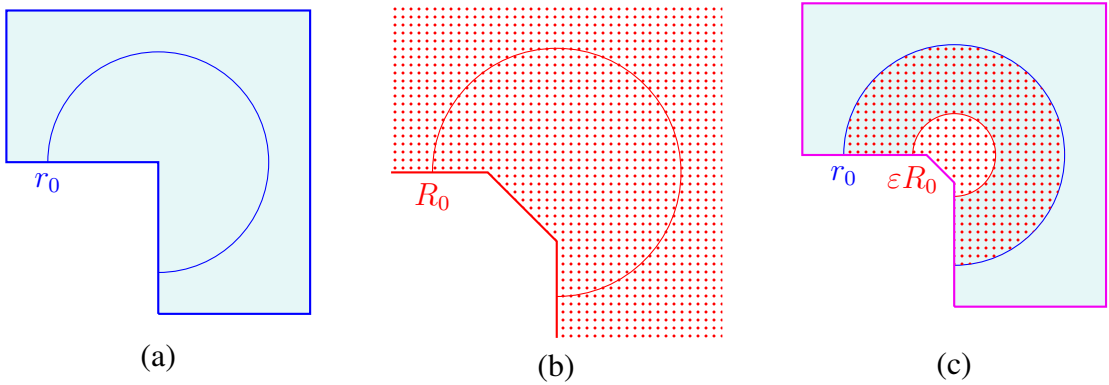

The regularity results in spaces $H^s$ proved in \cite{DBook:1988} (cf in particular Th.\,(18.13)) imply the following.

\begin{lemma}
\label{lem:Hs}
Let $s\ge1$ be such that $s - 1 < \frac{\pi}{\varpi_\ell}\,$ for all $\ell=1,\ldots,L$, i.e.
\begin{equation}
\label{eq:ell}
   s - 1 < \frac{\pi}{\varpi}\,.
\end{equation}
Then
\begin{enumerate}
\item If $f$ belongs to $\gF_\Omega\cap H^{s-2}(\Omega)$, then for all $\varepsilon<\varepsilon_0$, the solution $u_\varepsilon$ belongs to $H^s(\Omega_\varepsilon)$.
\item For all $j\ge1$, for $\rQ=\rP\cap \sB(R'_0)$ as in \eqref{eq:Q}, the function $K^+_j$ belongs to $H^s(\rQ)$ with estimates (for a constant $C$ independent of $j$, compare with \eqref{eq:estKj})
\[
   \Norm{K^+_j}{H^s(\rQ)} \le 
   C \DNorm{K^+_j}{H^{1}(\rP\cap\sB(2R_0))} .
\]
\end{enumerate}
\end{lemma}

As the assumptions on $\rP$ are covered by the framework of \cite{CoDa2Mu:2024}, a straightforward adaptation of the proof of Theorem \ref{th:Wmp} yields
\begin{theorem}
\label{th:Hs}
Let $s\ge1$ satisfy condition \eqref{eq:ell}.
Let $f\in\gF_\Omega\cap H^{s-2}(\Omega)$.
Then for any $\varepsilon\in(0,\varepsilon_0)$, the solution $u_\varepsilon$ of problem \eqref{eq:pbeps} belongs to $H^s(\Omega_\varepsilon)$. Moreover, for any $\varepsilon\in(0,\varepsilon_\star]$ it 
satisfies the inequality
\begin{equation}
\label{eq:ueHsQ}
   \Norm{u_\varepsilon}{H^s(\varepsilon\rQ)} \ge 
   \varepsilon^{\frac{\pi}{\omega}-s+1} \left( 
   |\mathbf{c}_1(u_0)|\, \Norm{K^+_1}{H^s(\rQ)}
   - \gamma_0\,\varepsilon^{\frac{\pi}{\omega}} \DNorm{f}{L^2(\Omega)}
   \right)
\end{equation}
where $\gamma_0$ is independent of $f$ and $\varepsilon$.
\end{theorem}

\begin{corollary}
If
\begin{equation}
\label{eq:l1Lell}
   \frac{\pi}{\omega} < s-1 < \frac{\pi}{\varpi}\,,
\end{equation}
and $\mathbf{c}_1(u_0)\neq0$ (condition \eqref{eq:c1}), the inequality \eqref{eq:ueHsQ} implies the blow-up of the seminorm $\Norm{u_\varepsilon}{H^s(\varepsilon\rQ)}$.
\end{corollary}
}

{
\subsection{Solution operators in polygons}
\label{ss:SolOp}
Let $\sU$ be a bounded polygon. The solution operator of problem \eqref{eq:1}
\[
\begin{array}{cccc}
   T : & H^{-1}(\sU) & \longrightarrow & H^1_0(\sU) \\
       & f &\longmapsto & u
\end{array}
\]
is an isomorphism. We denote it more specifically by $T_0(\sU)$, with its source and target spaces 
\[
   A_0(\sU) := H^{-1}(\sU)\quad\mbox{and}\quad B_0(\sU) := H^1_0(\sU).
\]

Concerning solutions in $H^2$,}
there is a remarkable identity valid on polygons $\sU$ (and more generally polyhedra)
\begin{equation}
\label{eq:H2}
   \forall u \in H^2\cap H^1_0(\sU),\quad
   \int_\sU |\Delta u |^2 = \sum_{|\alpha|=2} \int_\sU |\partial^\alpha u |^2\,,
\end{equation}
for which we can refer to Hanna \& Smith \cite{HannaSmith:1967} (and earlier to \cite{Aronszajn:1951}) and \cite{Grisvard:1972}. The assumption that $u$ belongs to $H^2(\sU)$ means that the coefficients of the singular functions associated with each non-convex corner of $\sU$ are all zero. 

{Set
\[
   A_1(\sU) := L^2(\sU)\quad\mbox{and}\quad B_1(\sU) := H^2\cap H^1_0(\sU).
\] 
If $\sU$ is a convex polygon, the solution operator of problem \eqref{eq:1} is an isomorphism $T_1(\sU):A_1(\sU)\to B_1(\sU)$. 
\smallskip

If $\sU$ has non-convex corners, there are two ways of defining solution operators at this level:
\begin{itemize}
\item By restriction of the source space, setting
\[
   A_1^\heartsuit(\sU) := \{f\in L^2(\sU),\quad Tf\in H^2(\sU)\}
\]
and keeping $B_1(\sU)$, which defines
\[
   T_1^\heartsuit(\sU) : A_1^\heartsuit(\sU) \longrightarrow B_1(\sU).
\]
\item By extension of the target space, setting
\[
   B_1^\spadesuit(\sU) := \{u\in H^1_0(\sU),\quad \Delta u\in L^2(\sU)\}
\]
and keeping $A_1(\sU)$, which defines
\[
   T_1^\spadesuit(\sU) : A_1(\sU) \longrightarrow B_1^\spadesuit(\sU).
\]
\end{itemize}
Both realizations $T_1^\heartsuit(\sU)$ and $T_1^\spadesuit(\sU)$ define isomorphisms: Concerning $T_1^\heartsuit(\sU)$ this is a consequence of the identity \eqref{eq:H2}, and for $T_1^\spadesuit(\sU)$ this is a consequence of \cite{Kond:1967,Grisvard:1972}.

Using functional interpolation, we can define accordingly for $\sigma\in(0,1)$ two realizations $T_\sigma^\heartsuit(\sU)$ and $T_\sigma^\spadesuit(\sU)$ acting between interpolation spaces 
\[
   T_\sigma^\heartsuit(\sU) : A_\sigma^\heartsuit(\sU) \longrightarrow B_\sigma(\sU)
   \quad\mbox{and}\quad
   T_\sigma^\spadesuit(\sU) : A_\sigma(\sU) \longrightarrow B_\sigma^\spadesuit(\sU),
\]
where
\[
   A_\sigma(\sU) = \big[A_1(\sU),A_0(\sU)\big]_{1-\sigma}\,, 
   \quad\quad
   A_\sigma^\heartsuit(\sU) = \big[A_1^\heartsuit(\sU),A_0(\sU)\big]_{1-\sigma}
\]
and
\[
   B_\sigma(\sU) = \big[B_1(\sU),B_0(\sU)\big]_{1-\sigma}\,, 
   \quad\quad
   B_\sigma^\spadesuit(\sU) = \big[B_1^\spadesuit(\sU),B_0(\sU)\big]_{1-\sigma}
\]
The standard interpolation spaces $A_\sigma(\sU)$ and $B_\sigma(\sU)$ are well known:
\begin{equation}
\label{eq:StInterp}
   B_\sigma(\sU) = H^{1+\sigma}\cap H^1_0(\sU)
   \quad\mbox{and}\quad
   A_\sigma(\sU) = H^{-1+\sigma}(\sU)
\end{equation}
if, for $\sigma=\frac{1}{2}$ one is careful to define the space $H^{-1/2}(\Omega)$  
as the dual space of the space of functions in $H^{1/2}(\sU)$ whose extension by $0$ belongs to $H^{1/2}(\R^2)$ (the latter space is sometimes known as $H^{1/2}_{00}(\Omega)$  \cite{LionsMagenes68}} or as $\widetilde H^{1/2}(\Omega)$ \cite{Triebel78}).

The less standard spaces $A_\sigma^\heartsuit(\sU)$ and $B_\sigma^\spadesuit(\sU)$ are more delicate and depend on openings at non-convex corners as we will see for the families $(\Omega_\varepsilon)$.

\subsection{Solution operators in families of polygons with approaching corners}
In the case of the family of polygons $\Omega_\varepsilon$ introduced in sec.\ \ref{ss:AppCor}
set
\[
   s_\Omega = 1+\frac{\pi}{\omega}\quad\mbox{and}\quad
   s_\rP = 1+\frac{\pi}{\varpi}\,.
\]
Then for all $s\in(1,s_\Omega)$, the Dirichlet problem defines an isomorphism from $H^1_0\cap H^s(\Omega)$ onto $H^{s-2}(\Omega)$. 

\subsubsection{Below the regularity threshold of the pattern domain}
Likewise, for all $\varepsilon\in(0,\varepsilon_0)$ and for all $s\in(1,s_\rP)$, the Dirichlet problem defines an isomorphism from $H^1_0\cap H^s(\Omega_\varepsilon)$ onto $H^{s-2}(\Omega_\varepsilon)$. 
This is proved in \cite{DBook:1988} for $s\ne\frac32$, and for $s=\frac32$ it then follows by a simple Hilbert space interpolation argument if one defines correctly the space $H^{-1/2}$ as mentioned above.

Rephrased using the interpolation spaces of sec.\ \ref{ss:SolOp}, we find that, when $1\le s<s_\rP$ and setting $\sigma=s-1$, the space $A_\sigma^\heartsuit(\Omega_\varepsilon)$ coincides with $H^{-1+\sigma}(\Omega_\varepsilon)$, which means in particular that the two operators $T_\sigma(\Omega_\varepsilon):H^{-1+\sigma}(\Omega_\varepsilon)\to H^{1+\sigma}\cap H^1_0(\Omega_\varepsilon)$ and $T_\sigma^\heartsuit(\Omega_\varepsilon)$ coincide.

But a consequence of Theorem \ref{th:Hs} is that the operator norm of $T_\sigma(\Omega_\varepsilon)$ blows up as $\varepsilon\to 0$: Indeed, set $f=\Delta(\varphi h^+_1)$, cf Lemma \ref{lem:2.1}, where in the present case
\[
   h^+_1 = r^{\frac{\pi}{\omega}} \sin\big(\frac{\pi\theta}{\omega}\big).
\]
By construction $u_0=\varphi h^+_1$, hence \eqref{eq:ueHsQ} holds for $u_\varepsilon$ with $\mathbf{c}_1(u_0)=1$.
Since the  $H^s(\Omega_\varepsilon)$ norm of $u_\varepsilon$ is larger than its seminorm in the region $\varepsilon\rQ$, we find
\[
\begin{aligned}
   |\me|\me| T_\sigma(\Omega_\varepsilon) |\me|\me| \ge 
   \frac{\DNorm{u_\varepsilon}{H^s(\Omega_\varepsilon)}}{\DNorm{f}{H^{s-2}(\Omega_\varepsilon)}} 
   &\ge
   \frac{\varepsilon^{\frac{\pi}{\omega}+1-s} \left( 
   \Norm{K^+_1}{H^s(\rQ)}
   - \gamma_0\,\varepsilon^{\frac{\pi}{\omega}} \DNorm{f}{L^2(\Omega)}
   \right)}{\DNorm{f}{H^{s-2}(\Omega_\varepsilon)}} \\
   &\ge
   \frac{\varepsilon^{\frac{\pi}{\omega}+1-s}  
   \Norm{K^+_1}{H^s(\rQ)}}{\DNorm{f}{H^{s-2}(\Omega_\varepsilon)}}
   - \gamma_0\,\varepsilon^{\frac{2\pi}{\omega}+1-s} 
\end{aligned}
\]
Since $\frac{2\pi}{\omega}+1-s>2-s>0$, and $\Norm{K^+_1}{H^s(\rQ)}>0$ we deduce that for all $s,\; s_\Omega<s<s_\rP$,
\begin{equation}
\label{eq:Tsig}
   |\me|\me| T_\sigma(\Omega_\varepsilon) |\me|\me| \ge 
   \gamma\, \varepsilon^{s_\Omega-s}\,,
\end{equation}
for $\varepsilon$ small enough with a $\gamma>0$ independent of $\varepsilon$.

By contrast, the basic estimate for an exact interpolation scale
yields
\[
   |\me|\me| T_\sigma^\heartsuit(\Omega_\varepsilon) |\me|\me| \;\le\;
   |\me|\me| T_1^\heartsuit(\Omega_\varepsilon) |\me|\me|^{\sigma} \;
   |\me|\me| T_0(\Omega_\varepsilon) |\me|\me|^{1-\sigma}.
\]
Since the domains $\Omega_\varepsilon$ are uniformly bounded polygons and owing to \eqref{eq:H2} we have {\em uniform bounds}
\[
   |\me|\me| T^\heartsuit_1(\Omega_\varepsilon) |\me|\me| \le M
   \quad\mbox{and}\quad
   |\me|\me| T_0(\Omega_\varepsilon) |\me|\me| \le M
\]
hence uniform bounds on $|\me|\me| T_\sigma^\heartsuit(\Omega_\varepsilon) |\me|\me|$ by interpolation. 

The apparent contradiction with \eqref{eq:Tsig} disappears when one understands that the operator norms $|\me|\me| T_\sigma(\Omega_\varepsilon) |\me|\me|$ and $|\me|\me| T_\sigma^\heartsuit(\Omega_\varepsilon) |\me|\me|$ are based on distinct norms in the spaces $H^{-1+\sigma}(\Omega_\varepsilon)\simeq A_\sigma^\heartsuit(\Omega_\varepsilon)$: The norm in $H^{-1+\sigma}(\Omega_\varepsilon)$ is the usual norm, while the norm in $A_\sigma^\heartsuit(\Omega_\varepsilon)$ is the interpolation norm of $\big[A_1^\heartsuit(\sU),A_0(\sU)\big]_{1-\sigma}$. We have with $f=\Delta(\varphi h^+_1)$ and $u_\varepsilon=T(\Omega_\varepsilon)f$
\[
   \frac{\DNorm{u_\varepsilon}{H^s(\Omega_\varepsilon)}}{\DNorm{f}{H^{s-2}(\Omega_\varepsilon)}}
   \ge \gamma\, \varepsilon^{s_\Omega-s}
   \quad\mbox{and}\quad
   \frac{\DNorm{u_\varepsilon}{B_\sigma^\heartsuit(\Omega_\varepsilon)}}
   {\DNorm{f}{A_\sigma^\heartsuit(\Omega_\varepsilon)}} \le M\,.
\]
Since for target spaces, by \eqref{eq:StInterp} the spaces $B_\sigma^\heartsuit(\Omega_\varepsilon)$ have a norm uniformly equivalent to the $H^s(\Omega_\varepsilon)$ norm we deduce (recall that $\sigma=s-1$)
\begin{equation}
\label{eq:fAH}
   \DNorm{f}{A_\sigma^\heartsuit(\Omega_\varepsilon)} \ge \gamma'\,
   \varepsilon^{s_\Omega-s} \DNorm{f}{H^{s-2}(\Omega_\varepsilon)}
\end{equation}
for a constant $\gamma'>0$ independent of $\varepsilon$.

When $s$ reaches the threshold $s_\rP$, the space $A_\sigma^\heartsuit(\Omega_\varepsilon)$ does not coincide any more with $H^{-1+\sigma}(\Omega_\varepsilon)$, and is not even closed in it. 
This phenomenon is in accordance with what has been observed in more general studies of interpolation of subspaces with finite codimension \cite{Asekritova:2015,Zerulla:2022}.
Results of the latter would, for $1\le s<s_{\rP}$, imply the equality of spaces 
$A_{s-1}^\heartsuit(\Omega_\varepsilon)=H^{s-2}(\Omega_\varepsilon)$ 
with equivalence of norms, but would not provide estimates of the constants in this equivalence which, as we have seen in \eqref{eq:fAH}, are not uniformly bounded as $\varepsilon\to0$ if $s_{\Omega}<s<s_{\rP}$.

\subsubsection{Above the regularity threshold of the pattern domain}
Assume for simplicity that all openings $\varpi_\ell$ are equal:
\[
   \varpi_\ell=\varpi,\quad \ell=1,\ldots,L,
\]
Hence $L(\varpi-\pi)=\omega-\pi$ by \eqref{eq:varpi}. When $s$ is above $s_\rP$, with still $s\le2$, the space $A_\sigma^\heartsuit(\Omega_\varepsilon)$ has the codimension $L$ in $H^{-1+\sigma}(\Omega_\varepsilon)$ and is defined by orthogonality conditions to $L$ dual singular functions. 

By contrast, the augmented space $B_\sigma^\spadesuit(\Omega_\varepsilon)$ is the direct sum of $H^{1+\sigma}(\Omega_\varepsilon)$ with $L$ independent singular functions.
Let us write this in detail for the solutions $u_\varepsilon$ when $f$ belongs to $\gF_\Omega$. 
The splitting of $u_\varepsilon$ in regular and singular parts can be deduced from the inner expansion \eqref{uepsinner.eq1} via the splitting of the canonical harmonic functions $K^+_j$ \eqref{eq:Kj} in $\rP$, which relies on the following result, \cite{Grisvard:1972,DBook:1988}:

\begin{lemma}
Let $s\in(s_\rP,2]$ and $0<\oR<\oR'$.
If $U$ belongs to $H^1(\rP\cap\sB(\oR'))$, satisfies the Dirichlet condition $U=0$ on $\partial\rP\cap\sB(\oR')$ and is such that $\Delta U\in H^{s-2}(\rP\cap\sB(\oR'))$, then there exist coefficients $D_\ell$ such that
\begin{equation}
\label{eq:singreg}
   U\dd{H^s}:=U - \sum_{\ell=1}^L D_\ell \Psi_\ell S_\ell \in H^s(\rP\cap\sB(\oR))
\end{equation}
and we have the estimates
\begin{equation}
\label{singregBRR'}
   \Norm{U\dd{H^s}}{H^{s}(\rP\cap\sB(\oR))} + \sum_{\ell=1}^L |D_\ell| \le 
   C\big( \DNorm{\Delta U}{H^{s-2}(\rP\cap\sB(\oR'))} 
   + \DNorm{U}{H^{1}(\rP\cap\sB(\oR'))} \big)\,.
\end{equation}
Here $\Psi_\ell$ is a cut-off function equal to $1$ in a neighborhood of the corner $O_\ell$ and to $0$ at the other corners, and 
\[
   S_\ell = R_\ell^{\frac{\pi}{\varpi}} \sin \frac{\pi\theta_\ell}{\varpi}
\]
in suitable polar coordinates $(R_\ell,\theta_\ell)$ centered at $O_\ell$. 
\end{lemma}

We apply this splitting to the canonical harmonic functions $K^+_j$ and obtain the existence of coefficients $D_{j,\ell}$ such that
\begin{equation}
\label{eq:KjHs}
   K^+\dd{j,H^s} := K^+_j - \sum_{\ell=1}^L D_{j,\ell} \Psi_\ell S_\ell \in H^s(\rQ)
\end{equation}
with estimates, owing to \eqref{eq:H1Kj} (with $\lambda^+_j=\frac{j\pi}{\omega}$)
\begin{equation}
\label{eq:Kjreg}
   \Norm{K^+\dd{j,H^s}}{H^{s}(\rQ)} + \sum_{\ell=1}^L |D_{j,\ell}| \le 
   C (\tfrac{j\pi}{\omega}+1)^{1/2}\, (2R_0)^{\frac{j\pi}{\omega}+1} \,.
\end{equation}

Combining this with the inner expansion \eqref{uepsinner.eq1}, we find the splitting of $u_\varepsilon$
written for $x\in\Omega_\varepsilon\cap\sB(r_0/2)$ as (with $r_\ell=\varepsilon R_\ell$):
\begin{subequations}
\begin{equation}
\label{eq:Splitue}
   u_\varepsilon(x) = (u_\varepsilon)\dd{H^s}(x) + \sum_{\ell=1}^L d_{\varepsilon,\ell}
   \Psi_\ell\left(\frac{x}{\varepsilon}\right) S_\ell(r_\ell,\theta_\ell)
\end{equation}
with the regular part $(u_\varepsilon)\dd{H^s}$ and the coefficients $d_{\varepsilon,\ell}$ given by the
converging expansions
\begin{align}
\label{eq:innepsHs}
   (u_\varepsilon)\dd{H^s}(x) &= 
   \sum_{j\ge 1} \sum_{\sfe\in\sfE^\infty} \varepsilon^{\sfe+\lambda^+_j}
   \mathbf{c}_j(\sfu_\sfe)\, K^+\dd{j,H^s}\!\left(\frac{x}{\varepsilon}\right)\,, \\
   d_{\varepsilon,\ell} &=
   \sum_{j\ge 1} \sum_{\sfe\in\sfE^\infty} \varepsilon^{\sfe+\lambda^+_j}
   \mathbf{c}_j(\sfu_\sfe)\, D_{j,\ell} , \quad \ell=1,\ldots,L.
\label{eq:innepsd}
\end{align}
\end{subequations}
Following the same method as in the proof of Theorem \ref{th:Wmp}
we detach the principal contribution, coming from the term $j=1$, $\sfe=0$, and find:
\begin{subequations}
\begin{equation}
\label{eq:ueHsreg}
   \Norm{(u_\varepsilon)\dd{H^s}}{H^s(\varepsilon\rQ)} \ge 
   \varepsilon^{\frac{\pi}{\omega}+1-s} \left( 
   |\mathbf{c}_1(\sfu_0)|\, \Norm{K^+\dd{1,H^s}}{H^{s}(\rQ)}
   - \gamma_0\,\varepsilon^{\frac{\pi}{\omega}} \DNorm{f}{L^2(\Omega)}
   \right)
\end{equation}
where $\gamma_0$ is independent of $f$ and $\varepsilon$ (compare with \eqref{eq:ueHsQ}), and
\begin{equation}
   d_{\varepsilon,\ell} = 
   \varepsilon^{\frac{\pi}{\omega}-\frac{\pi}{\varpi}} \mathbf{c}_1(\sfu_0)\,D_{1,\ell}
   + \cO(\varepsilon^{\frac{2\pi}{\omega}-\frac{\pi}{\varpi}}).
\end{equation}
\end{subequations}
A similar behavior of singularity coefficients has been observed in \cite{DaToVi10}.

\begin{theorem}
\label{th:s>s_Omega}
Let $s\in(s_\rP,2]$ and $f\in \gF_\Omega$. 
If $\mathbf{c}_1(u_0)\ne0$, then both the regular part 
$(u_\varepsilon)\dd{H^s}$ and the coefficients $d_{\varepsilon,\ell}$ blow up as $\varepsilon\to0$ according to
\[
 \Norm{(u_\varepsilon)\dd{H^s}}{H^s(\varepsilon\rQ)} \ge 
   \gamma\varepsilon^{\frac{\pi}{\omega}+1-s}\;;\qquad
 |d_{\varepsilon,\ell}| \ge
   \gamma\varepsilon^{\frac{\pi}{\omega}-\frac{\pi}{\varpi}}
\]
with a constant $\gamma$ independent of $\varepsilon$.
\end{theorem}

\begin{proof}
It remains to show that $D_{1,\ell}\ne0$. This follows from the maximum principle, compare Remark~\ref{R:pos}, applied to the non-negative harmonic function $K_{1}^{+}$. The latter is non negative, because it is continuous on $\overline{\rP}$ and behaves asymptotically like $h_{1}^{+}$ at infinity. 
\end{proof}

Conversely, if a right-hand side $f\in \gF_\Omega$ is such that $u_\varepsilon$ belongs to $H^2(\Omega_\varepsilon)$ for some value of $\varepsilon$, we find an upper bound for the singularity coefficient $\mathbf{c}_1(u_0)$ of the solution $u_0$ of the limiting problem \eqref{eq:pbeps} in $\Omega$. 

\begin{theorem}
\label{th:H2}
There exists a constant $C_{\Omega,\rP}$ (only depending on $\Omega$ and $\rP$) such that the following holds: If $f\in \gF_\Omega$ and $\varepsilon\in(0,\varepsilon_\star]$ are such that
\[
   f\in A_1^\heartsuit(\Omega_{\varepsilon}),\quad\mbox{i.e.}\quad
   u_\varepsilon\in H^2(\Omega_\varepsilon),
\]
then there holds
\begin{equation}
\label{eq:c1bis}
   |\mathbf{c}_1(u_0)| 
   \le C_{\Omega,\rP}\, \varepsilon^{1-\frac{\pi}{\omega}} \DNorm{f}{L^2(\Omega)} \,.
\end{equation}
\end{theorem}

\begin{proof}
The assumption on $u_\varepsilon$ means that $u_\varepsilon = (u_\varepsilon)\dd{H^2}$. Then \eqref{eq:ueHsreg} yields
\[
   \Norm{u_\varepsilon}{H^2(\varepsilon\rQ)} \ge 
   \varepsilon^{\frac{\pi}{\omega}-1} \left( 
   |\mathbf{c}_1(u_0)|\, \Norm{K^+\dd{1,H^2}}{H^{2}(\rQ)}
   - \gamma_0\,\varepsilon^{\frac{\pi}{\omega}} \DNorm{f}{L^2(\Omega)}
   \right),
\]
hence the inequality
\[
   |\mathbf{c}_1(u_0)|\, \Norm{K^+\dd{1,H^2}}{H^{2}(\rQ)}
   \le
   \varepsilon^{1-\frac{\pi}{\omega}} 
   \DNorm{u_\varepsilon}{H^2(\Omega_\varepsilon)} +
   \gamma_0\,\varepsilon^{\frac{\pi}{\omega}} \DNorm{f}{L^2(\Omega)}.
\]
Since \eqref{eq:H2} implies the uniform bound $\DNorm{u_\varepsilon}{H^2(\Omega_\varepsilon)}\le M\DNorm{f}{L^2(\Omega)}$, the theorem is proved.
\end{proof}

\section*{Statement and Declaration}
The authors acknowledge the support of the Centre Henri Lebesgue
ANR-11-LABX-0020-01.\\
The authors declare that they have no competing interests.\\
No data were generated.\\
Both authors contributed to the ideas, contents, and formulation of this work.


\end{document}